\documentclass[reqno]{amsart}

\usepackage{amsmath,amssymb,amsthm}
\usepackage{graphicx,tikz}

\newtheorem*{thm}{Theorem}
\newtheorem*{proposition}{Proposition}

\newtheorem*{corollary}{Corollary}

\newtheorem{lemma}{Lemma}

\theoremstyle{plain}

\theoremstyle{plain}

\begin{document}

\title[]{Concentration of Hitting Times\\ in Erd\H{o}s-R\'enyi graphs}

\author[]{Andrea Ottolini}
\address[]{Department of Mathematics, University of Washington, Seattle, WA 98195, USA}
\email{ottolini@uw.edu}
\email{steinerb@uw.edu}

\author[]{Stefan Steinerberger}

\thanks{S.S. was supported by the NSF (DMS-2123224) and the Alfred P. Sloan Foundation.}

\begin{abstract} We consider Erd\H{o}s-R\'enyi graphs $G(n,p)$ for $0 < p < 1$ fixed and $n \rightarrow \infty$ and study the expected number of steps, $H_{wv}$, that a random walk started in $w$ needs to first arrive in $v$. A natural guess is that an Erd\H{o}s-R\'enyi random graph is so homogeneous that it does not really distinguish between vertices and  $H_{wv} = (1+o(1)) n$. L\"owe-Terveer established a CLT for the Mean Starting Hitting Time suggesting $H_{w v} =  n \pm \mathcal{O}(\sqrt{n})$. We prove the existence of a strong concentration phenomenon: $H_{w v}$ is given, up to a very small error of size $\lesssim (\log{n})^{3/2}/\sqrt{n}$, by an explicit simple formula involving only the total number of edges $|E|$, the degree $\deg(v)$ and the distance $d(v,w)$.
\end{abstract}

\maketitle

\vspace{20pt}

\section{Introduction and Results}

\subsection{The Phenomenon.} The purpose of this paper is to demonstrate a very strong concentration phenomenon for hitting times on Erd\H{o}s-R\'enyi random graphs. Given a realization of the graph, the hitting time $H_{wv}$ is the expected time that a simple random walk started in vertex $w$ needs to reach the vertex $v$ for the first time. We consider the distribution of hitting times $H_{wv}$ when the graph $G=G(n,p)$ is an Erd\H{o}s-R\'enyi random graph for $0 < p < 1$ fixed and $n \rightarrow \infty$. 
 The phenomenon we are interested in can be illustrated with a simple example, see Fig. 1.

\vspace{5pt}

\begin{center}
\begin{figure}[h!]
\begin{tikzpicture}[scale=1]
\node at (0,0) {\includegraphics[width=0.48\textwidth]{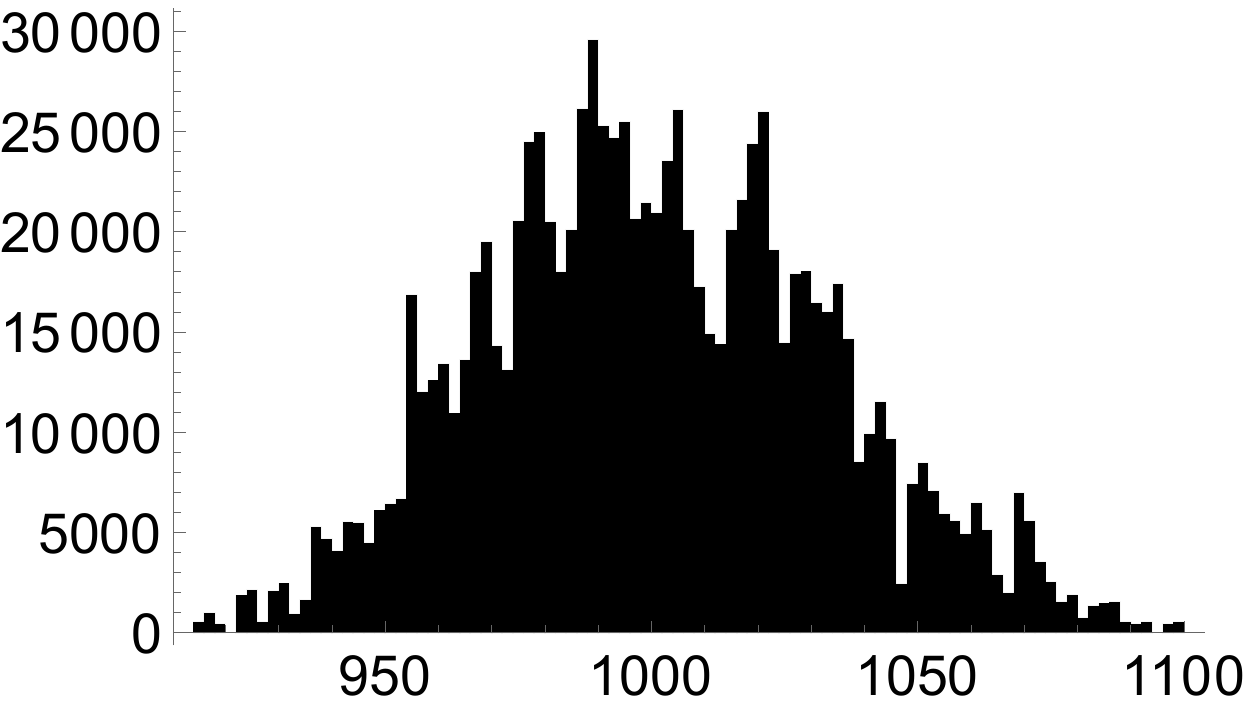}};
\node at (6.2,0) {\includegraphics[width=0.48\textwidth]{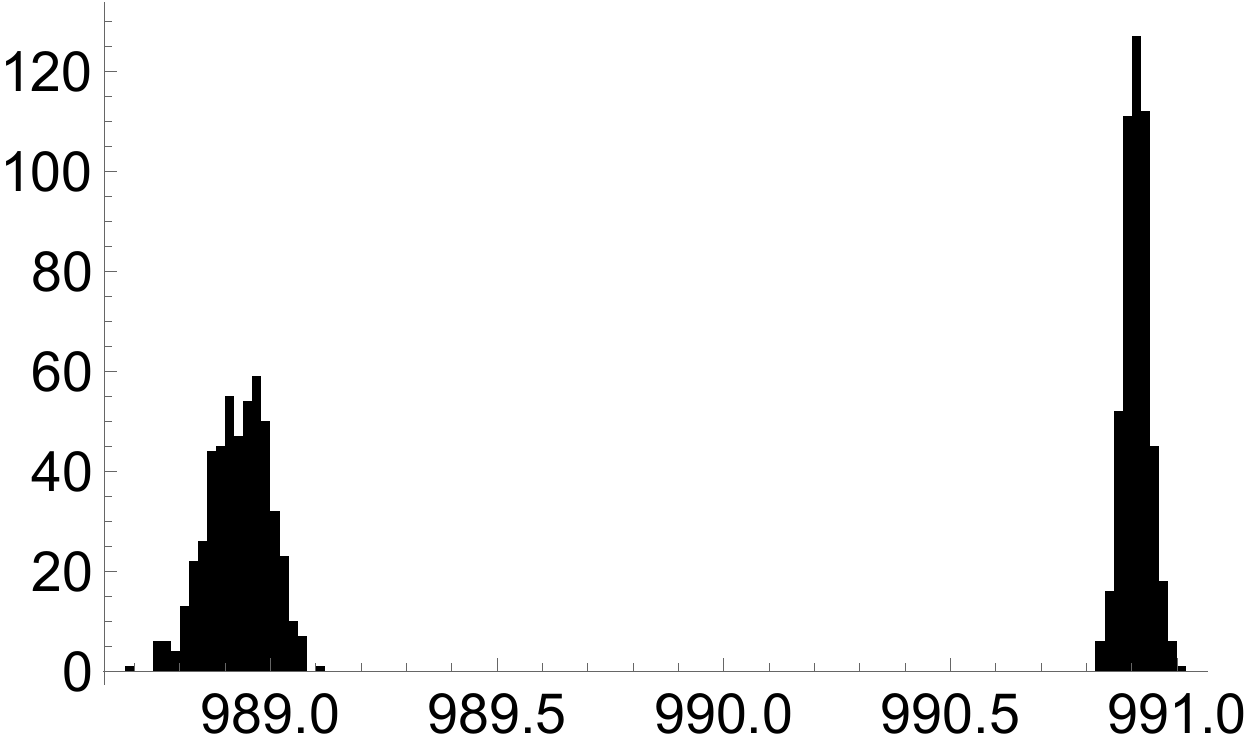}};
\end{tikzpicture}
\caption{Left: the distribution of all hitting times $H_{ij}$ for a realization of $G(1000,1/2)$. Right: the distribution of hitting times $H_{wv}$ for $w \neq v$ for a fixed vertex $v$. This quantity is tightly concentrated in two very short intervals.}
\end{figure}
\end{center}

 A natural first guess is that such a random walk is rapidly mixing and should spend roughly equal amounts of time in each vertex and all the vertices should, at least to leading order, be indistinguishable. This would suggest that $H_{wv} \sim n$ whenever $w \neq v$ and this was also suggested in the physics literature \cite{sood}. A version of this scaling result was rigorously proven by L\"owe \& Torres \cite{lowe}. If we denote by $H_{\mu v}:=\sum \mu(w)H_{wv}$ the hitting time starting from a point distributed according to a probability measure $\mu$, then $H_{\pi v} = (1+o(1)) n$ where $\pi(w):=\deg(w)/(2|E|)$ is the (random) stationary distribution. We also refer to subsequent results of Helali \& L\"owe \cite{hel}, von Luxburg, Radl \& Hein \cite{von} and Sylvester \cite{sylvester}.
A look at some numerical examples (see Fig. 1) suggests that the typical deviation of $H_{wv}$ from the mean appears to be on the order of $\sqrt{n}$. One would naturally expect the existence of a central limit theorem. Such a result was recently established by L\"owe \& Terveer \cite{lowe2}, who showed that $H_{\pi v} -n$, suitably rescaled, converges to a Gaussian (in distribution).
This, while not a statement about $H_{wv}$ in itself, does indicate that averaged hitting times have Gaussian fluctuations. 

\subsection{Main Result.} We provide what is essentially an explicit formula for $H_{wv}$
up to a very small error term that tends to 0 as $n \rightarrow \infty$.

\begin{thm}\label{mainresult}
Let $G(n,p)$ be an Erd\H{o}s-R\'enyi random graph with $0 < p < 1$. Then, as $n \rightarrow \infty$, we have with high probability that for any vertex $v$ and all vertices $w \neq v$
$$ H_{wv} = \frac{2|E|}{\deg(v)} +
 \begin{cases}
-1 \quad &\mbox{if}~(w,v) \in E \\
-1 + 1/p \quad &\mbox{if}~(w,v) \notin E 
\end{cases} \quad
+ \mathcal{O}\left(\frac{(\log{n})^{3/2}}{\sqrt{n}} \right)$$
where the implicit constant in the error term depends only on $p$. 
\end{thm}
This explains the second picture in Fig. 1, we refer to Fig. 2 for more examples.
 We also prove (Lemma 1) that the average hitting time starting from vertices $w$ that are adjacent to $v$ is exactly $2|E|/\deg(v) - 1$. The Theorem implies that the central limit theorem obtained by L\"owe-Terveer \cite{lowe2} for $H_{\pi \cdot}$, the hitting time from a stationary initial point, holds for $H_{\mu v}$ for any starting measure $\mu$, and in particular it also holds for the hitting times $H_{w v}$ themselves, see \S 2.8. 

\begin{center}
\begin{figure}[h!]
\begin{tikzpicture}
\node at (0,0) {\includegraphics[width=0.45\textwidth]{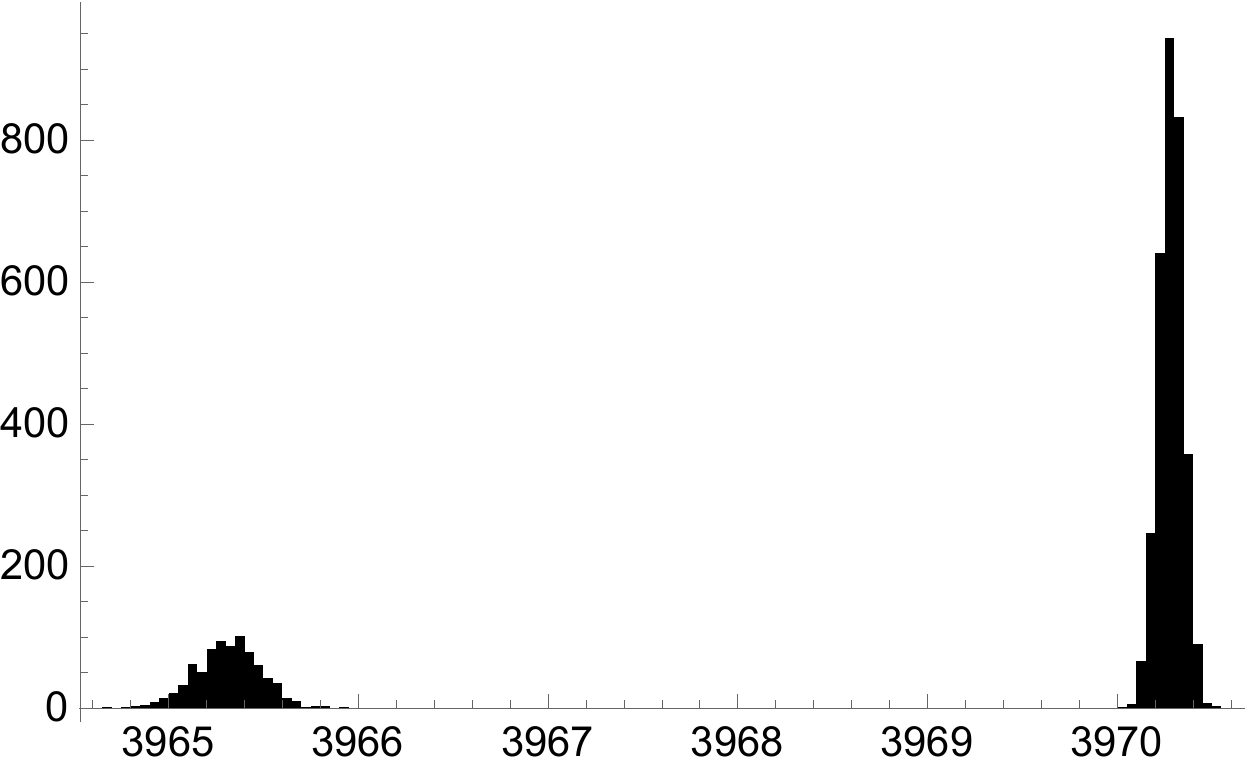}};
\node at (6,0) {\includegraphics[width=0.45\textwidth]{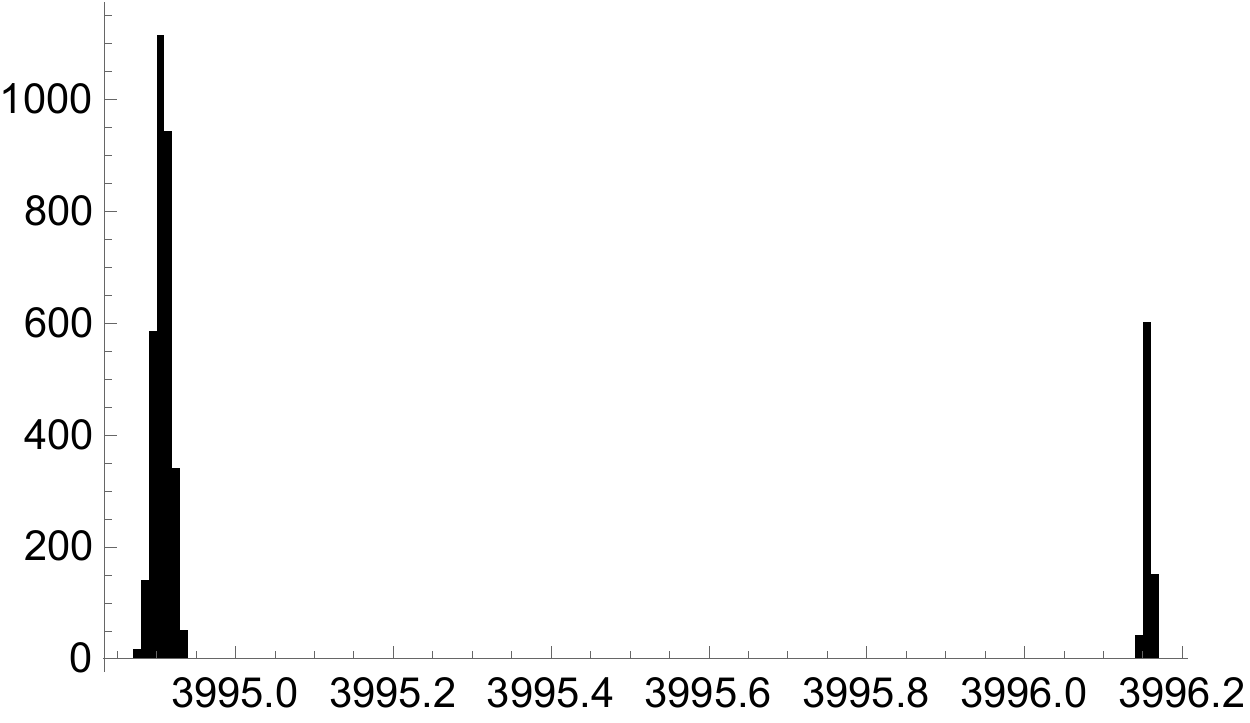}};
\end{tikzpicture}
\caption{Left: $G(4000, 0.2)$ with predicted localization at $2|E|/\deg(v) - 1 = 3965.3$ and $2|E|/\deg(v) +4 = 3970.3$. Right: $G(4000, 0.8)$ with predicted localization at $ 3994.9$ and $3996.16$. }
\end{figure}
\end{center}

One interpretation, which also informs the structure of the proof, is that an Erd\H{o}s-R\'enyi random graph has diameter 2 (with high probability) and can be understood as behaving a little bit like a three-state Markov chain: the vertex $v$ under consideration, the vertices at distance 1 from $v$, and the vertices at distance 2 from $v$. Note that transitions to $v$ are rare, given that a typical vertex has an high degree. During that time, the mixing properties of the graph are of sufficiently high quality that they essentially lead to near-deterministic behavior. However, it is important whether one starts at distance 1 or 2 from the target vertex $v$ (not surprising since we aim to determine the hitting term up to an error of $o(1)$). In fact, since transitions from most vertices at distance $2$ to the set of vertices at distance $1$ happen at rate roughly $\sim p$, the correction $1/p$ can be understood as a geometric hitting time. It is conceivable that the result has an extension when $p$ decays with $n$, generalizing accordingly the number of possible states since the diameter will increase with high probability (see, e.g., \cite{chu} for a variety of regimes). Moreover, while some features of our proof rely on the Erd\H{o}s-R\'enyi structure, a similar approach may lead to less refined results (e.g., a law of large number or a central limit theorem) for other models of random graphs, provided that the mixing time is sufficiently small compared to the hitting time, bypassing a spectral approach. This seems like a promising avenue for further research.

\subsection{Implications.}
The result has a number of implications in terms of known results.  L\"owe \& Torres \cite{lowe} prove, using a spectral theory approach, that the averaged hitting time $H_{\pi v}$ satisfies
$$ H_{\pi v} \geq \frac{2|E|}{\deg(v)} - 2$$
which can now be seen to be almost optimal. Our Theorem implies
\begin{align*}
 H_{\pi v} = \sum_{w} \pi(w) H_{wv} = \frac{2|E|}{\deg(v)} - 3 + \frac{1}{p} + \mathcal{O}\left(\frac{(\log{n})^{3/2}}{\sqrt{n}} \right).
\end{align*}
The result also has immediate implications for the commute time
$ \kappa_{wv} = H_{wv} + H_{vw}$
and the effective resistance
$$ \sigma_{wv} = \frac{\kappa_{wv}}{ 2 |E|}.$$ 
For example, we can deduce a very precise asymptotic expansion
$$ \sigma_{wv} = \frac{1}{\deg(v)} + \frac{1}{\deg(w)} +  \frac{2}{n^2 p}\begin{cases}
-1 \quad &\mbox{if}~(v,w) \in E \\
-1 + 1/p \quad &\mbox{if}~(v,w) \notin E 
\end{cases} 
+ \mathcal{O}\left(\frac{(\log{n})^{3/2}}{n^{5/2}} \right).$$ 
We remark that the effective resistance is intimately connected to a number of classical objects in probability, such as uniform spanning trees (a result dating back to Kirchhoff \cite{kir}), and it is easy to rephrase the asymptotic above correspondingly (see, e.g., Chapter $4$ in \cite{lyo}). Comparing to arbitrary graphs, a result of Lov\'asz \cite{lov} shows that the commute time always satisfies
$$ |E| \left( \frac{1}{\deg(v)} + \frac{1}{\deg(w)} \right) \leq \kappa_{wv} \leq  \frac{2 |E|}{1- \lambda_2} \left( \frac{1}{\deg(v)} + \frac{1}{\deg(w)} \right),$$
where $\lambda_2$ is the second eigenvalue of the transition matrix. For a typical Erd\H{o}s-R\'enyi random graph, our result says that the lower bound is almost saturated. We conclude with some spectral implication (see L\"owe \& Torres \cite{lowe} or Lov\'asz \cite{lov} for additional details). Let $A \in \left\{0,1 \right\}^{n \times n}$ denote the adjacency matrix of the graph and $D$ denote the degree matrix of the graph, we can introduce $B = D^{-1/2} A D^{-1/2}$. $B$ has largest eigenvalue $\lambda_1 = 1$ and has eigenvalues $\lambda_1 \geq \lambda_2 \geq \dots \geq \lambda_n$. We use $\sigma_k$ to denote the eigenvectors of length 1 associate to the $k-$th eigenvalue. Then, see L\"owe \& Torres \cite{lowe} or Lov\'asz \cite{lov}, we have
$$ H_{\pi v} = \frac{2|E|}{\deg(v)} \sum_{w=2}^{n} \frac{\sigma_{wv}^2}{1 - \lambda_k}.$$
Since $H_{\pi v}$ is nearly constant in $v$, this can be understood as a type of weighted equidistribution result for the eigenvectors of $B$.
Finally, our result allows for estimates related to the quasi-stationary distribution $\tilde \pi_v$ associated to a vertex $v$, that is to say, the Perron-Frobenius eigenvector of the sub-stochastic matrix obtained by neglecting the row and column corresponding to $v$ in the transition matrix (see \cite{persi} and references therein for a discussion on its use). It is a classical fact that its eigenvalue $\lambda_v$ satisfies
\begin{align*}
H_{\tilde \pi_v v}=\frac{1}{1-\lambda_v}, 
\end{align*}
and so our result gives two-sided bounds on $\lambda_v$ that are uniform in $v\in V$ (w.h.p.).

\section{Proof}

\subsection{Outline} It is a classical result (see, for example, \cite{klee}) that with high probability the diameter of $G(n,p)$ for fixed $p$ is 2. From now on, we will work on the event that this occurs. We fix a vertex $v \in V$ and consider from now on only the (first) hitting time of $v$ of a random walk starting in any vertex $w \in V$. Naturally, $H_{vv} = 0$.
For any other vertex $w \neq v$, we have
\begin{align} \label{eq:one}
 H_{wv} = 1+\frac{1}{\deg(w)} \sum_{(w,u) \in E} H_{uv}.
 \end{align} 

 Having fixed $v \in V$, we denote the set of all vertices at distance 1 from $v$ by $A$,
 $$ A = \left\{w \in V: d(v,w) = 1\right\}$$
 and the set of all vertices at distance 2 from $v$ by $B = V \setminus \left( \left\{ v \right\} \cup A \right)$.
We first establish concentration of the hitting times in $A$ and $B$ and then use the concentration to deduce everything else. The proof works as follows.
\begin{enumerate}
\item We first recall some basic facts from probability theory (in \S 2.2) that will be used throughout to argument: concentration of the degrees, a technical result on the convergence rate, and an exact computation for the hitting time when the starting distribution is uniform in $A$. \item In \S 2.3 we derive a combinatorial formula for the hitting time. It is then shown that many of the terms in the formula are tightly concentrated, conditional on a certain a priori bound to be established.
\item \S 2.4 shows the a priori bound
 $$H_{wv} = n + \mathcal{O}(\sqrt{n \log{n}})$$
 with high probability. This result is not as good as the main result but shows that certain terms that arise in the formula from \S 2.3 are more tightly concentrated than one might expect from $H_{wv} = (1+o(1))\cdot n$. 
\item \S 2.5 uses all the preceeding arguments to show strong concentration of the hitting time in the set $B$,
$$ \max_{w_1, w_2 \in B} | H_{w_1v} - H_{w_2v} | \lesssim \frac{(\log{n})^{3/2}}{\sqrt{n}}.$$
\item In \S 2.6 a similar argument is repeated for vertices in $A$ and
$$ \max_{w_1, w_2 \in A} | H_{w_1v} - H_{w_2v} | \lesssim \frac{(\log{n})^{3/2}}{\sqrt{n}}.$$
\item Having shown concentration in both $A$ and $B$, it is shown in \S 2.7 that
the difference in expectation between the hitting times in $A$ and $B$ is $1/p + \mathcal{O}(\sqrt{\log{n}}/\sqrt{n})$. Since we already know (Lemma 1) the expected hitting time for vertices in $A$, the result follows.
\item \S 2.8 observes an associated Central Limit Theorem immediately follows, \S 2.9 proves the technical Proposition introduced in \S 2.2.
\end{enumerate}

\subsection{Preliminary facts}
We first collect a couple of elementary facts that will be used several times. We will often omit w.h.p. (with high probability) in what follows: every time we invoke a property of Erd\H{o}s-R\'enyi random graphs it is to be understood as being valid with high probability. The sequence of degrees are Bernoulli random variables $\deg (v) \sim B(n,p)$, with
$ \mathbb{E}\deg(v) = (n-1)p$. While they are not independent, a union bound, together with a Chernoff bound for binomial random variables, gives
$$np - c_p  \sqrt{n} \sqrt{\log{n}} \leq \min_{v\in V} \deg(v) \leq \max_{v\in V} \deg(v) \leq np + c_{p} \sqrt{n} \sqrt{\log{n}}$$
We note that more precise results, in particular about the size of $c_p$, are known (see, e.g., \cite{gestalt}), though we will not try to optimize over that as they are not required for the remainder of the argument. In particular, this gives
\begin{equation}\label{boundegree}
\max_{v \in V} \left| \deg(v) - np \right| \lesssim \sqrt{n \log{n}} 
\end{equation}
with high likelihood and with an implicit constant depending on $p$ only.\\

Another ingredient that we will use several times is the following technical proposition about the behavior of a random walk on an Erd\H{o}s-Renyi graph.
\begin{proposition} Let $G=G(n,p)$ be an Erd\H{o}s-Renyi graph with $0 < p < 1$ fixed and $n \rightarrow \infty$. 
Let $v$ be an arbitrary vertex, $\mu_0 = \delta_{v}$ and $\mu_{k+1} = \mu_k D^{-1}A$ be the probability distribution after $k$ steps of the random walk. Using $\pi$ to denote the stationary distribution, for $k \geq 1$,
$$ \| \mu_k - \pi\|_{\ell^2} \lesssim_{k,p} \frac{(\log n)^{\frac{k-1}{2}}}{n^{k/2}}.$$
\end{proposition}
 The Cauchy-Schwarz inequality implies that
\begin{align}\label{specgap}
\| \mu_k-\pi \|_1 &\leq \sqrt n\cdot \| \mu_k-\pi \|_2 \lesssim_{k,p}  \left( \frac{\log{n}}{n}\right)^{\frac{k-1}{2}}. 
\end{align}
We will only use his argument for $k=3$. The argument is self-contained and presented at the end of the paper. 
Our argument will benefit from having at least one probability measure $\nu$ for which we can obtain explicit estimates on the expected hitting time. There is a particularly simple choice.

\begin{lemma}\label{exact} Let $v \in V$ be arbitrary and let $A$ be the set of neighbors of $v$.
Consider the probability measure $\nu= 1_{A}/\deg(v)$. Then, we have
\begin{align*}
H_{\nu v}=\frac{2|E|}{\deg(v)}-1.
\end{align*}
\end{lemma}
\begin{proof}
Consider a random walker starting from the stationary distribution $\pi$ given by $\pi(w) = \deg(w)/(2 |E|)$. There are two cases: either the initial vertex is already $v$ (and then the hitting time is 0) or it is not, in which case the new probability distribution after one step is distributed according to some new measure $\mu$. On the other hand, conditional on the starting point being $v$, the distribution after one step would be exactly $\nu$. Therefore, we have 
\begin{align*}
\mu=\frac{\pi-\pi(v)\nu}{1-\pi(v)}.
\end{align*}
In particular, we obtain
\begin{align*}
H_{\pi v}=(1-\pi(v))\left(1+H_{\mu v}\right)=
        1-\pi(v)+H_{\pi v}-\pi(v)H_{\nu v}.
\end{align*}
Rearranging, this gives 
\begin{align*}
H_{\nu v}=\frac{1}{\pi(v)}-1=\frac{2|E|}{\deg(v)}-1.
\end{align*}
\end{proof}
The result has an electrical network interpretation, which can be turned into a proof when $A=\{w\}$ is a single vertex connected to $v$ through $\deg(v)$ edges. The effective resistance $\kappa_{wv}$ is that of $\deg(v)$ resistances in parallel. Then,
\begin{align*}
\frac{H_{wv}+H_{vw}}{2E}=\kappa_{vw}=\frac{1}{\deg(v)},
\end{align*}
and the result follows since $H_{vw}=1$. 
\subsection{A formula for the Hitting time for $w \in B$.}
Let us now fix a vertex $w \in B$.  As an initial motivation, we consider the distribution of a random walk started in $w \in B$ with a length of 2 steps.
We start with the basic fact (introduced above as \eqref{eq:one}) that for a vertex $w \neq v$
\begin{align*}
 H_{wv} = 1+\frac{1}{\deg(w)} \sum_{(w,u) \in E} H_{uv}.
 \end{align*}
Since $w \in B$ and $d(w,v) = 2$, we can apply the identity twice. If we denote the arising distribution after 2 steps
by $\nu_w$, then this implies the exact equation
$$ H_{wv} = 2 + H_{\nu_w v}.$$
For the remainder of the argument, we will consider the distribution of random walks started in $w \in B$ after 3 steps and
this probability distribution will be denoted by $\mu_w$.
The analogous equation
\begin{align} \label{eq:two}
 H_{wv} = 3 + H_{\mu_wv} \qquad \mbox{is false} 
\end{align}
because \eqref{eq:one} is only valid in vertices $w \neq v$ but a random walk started in $w$ might end up in $v$ after 2 steps.
We also note that this is the only source of error. If the random walk ends up in $v$ after 2 steps, then it is uniformly distributed in $A$
after 3 steps and the equation \eqref{eq:two} mistakenly contributes
$$ \mathbb{P}\left( \mbox{random walk is in}~v~\mbox{after 2 steps}\right) \cdot  \frac{1}{|A|} \sum_{w \in A} H_{wv}$$
while the true hitting time would have been 2. We note that the sum, the average hitting time for a starting vertex in $A$, has already been computed above in Lemma 1 but we will not actually need its precise value at this point in the argument.
Correcting for this, we arrive at
\begin{align} \label{eq:three}
 H_{wv} &= 3 + \sum_{a \in V} \mu_w(a) \cdot H_{av} \\
 &+  \mathbb{P}\left( \mbox{random walk is in}~v~\mbox{after 2 steps}\right) \cdot \left(-1 -  \frac{1}{|A|} \sum_{a \in A} H_{av} \right). \nonumber
\end{align}

We will now argue that \eqref{eq:three} can be used to show that the hitting time is tightly concentrated by showing that each of its terms is tightly concentrated.
We start by noting that Lemma 1 implies that
$$ \left(-1 -  \frac{1}{|A|} \sum_{a \in A} H_{av} \right) = - \frac{2|E|}{\deg(v)}$$
is independent of $w$.
We continue by computing the probability of a random walk started in $w$ hitting $v$ within 2 steps. For this to happen, the first step of the random walk has to lead to $A$ and the second step has to lead from $A$ to $v$. The likelihood of moving from $w$ to $A$ is simply the number of neighbors that $w$ has in $A$ compared to its total degree. Introducing the abbreviation
$$ N_A(w) = \left\{ a \in A: (a,w) \in E \right\}$$
for the number of neighbors a vertex has in $A$, we have that the likelihood of moving from $w$ to $A$ is given by
$$ \frac{\# N_A(w)}{\deg(w)} = \frac{|A|}{n} + \mathcal{O}\left(\frac{\sqrt{\log{n}}}{\sqrt{n}}\right) \qquad \mbox{with high probability.}$$
owing to \eqref{boundegree}. The second ingredient is the likelihood of moving from $A$ to $v$ in the second step which, conditioned on the first step leading from $w$ to $A$, happens with likelihood
$$\frac{1}{\# N_A(w)} \sum_{z \in N_A(w)} \frac{1}{\deg(z)}.$$
This is an average over inverse degrees over the set of neighbors of $w$ that are in $A$. Abbreviating
$$ X = \left| \frac{1}{\# N_A(w)} \sum_{z \in N_A(w)} \frac{1}{\deg(z)}  - \frac{1}{(n-1)p} \right|$$
we obtain, appealing \eqref{boundegree} that (w.h.p.)
\begin{align*}
X &\leq \max_{v\in V}\left|\frac{1}{\deg(v)}-\frac{1}{(n-1)p}\right|\\
&\leq \left| \frac{1}{np \pm \mathcal{O}(\sqrt{n \log{n}})} - \frac{1}{np} \right| \lesssim \frac{ \sqrt{n \log{n}}}{n^2 p^2} \lesssim \frac{\sqrt{\log{n}}}{n^{3/2}}.
\end{align*}
We see that $\mathbb{P} = \mathbb{P}\left( \mbox{random walk in}~v~\mbox{after 2 steps}\right)$ satisfies
\begin{align} \label{eq:four}
\mathbb{P} &= \left( \frac{|A|}{n} + \mathcal{O}\left(\frac{\sqrt{\log{n}}}{\sqrt{n}}\right)  \right) \left( \frac{1}{np} + \mathcal{O}\left(\frac{\sqrt{\log{n}}}{n^{3/2}} \right)\right)  \nonumber \\
&= \frac{|A|}{n^2 p}  +  \mathcal{O}\left(\frac{\sqrt{\log{n}}}{ n^{3/2}}\right).  
\end{align}
Therefore, \eqref{eq:three} can be simplified as
$$  H_{wv} = 3 + \sum_{a \in V} \mu_w(a) \cdot H_{av} + \mbox{Error}(w)$$
where
$$ \mbox{Error}(w) = - \frac{2 |E|}{n^2 p} \pm \mathcal{O}\left( \frac{\sqrt{\log{n}}}{\sqrt{n}} \right)$$
and the error term is uniform over all choices of $w \in B$.

\subsection{A Cheap Uniform Bound}
The purpose of this short section is to deduce a cheap uniform bound. This estimate will then be bootstrapped in the subsequent sections to obtain the main result.
%
\begin{lemma} \label{lem:1} Let $G=G(n,p)$ with $0 < p < 1$. Then, as $n \rightarrow \infty$, with high probability 
$$ \max_{v\neq w \in V} ~\left|  H_{wv} - n \right| \lesssim \sqrt{n \log{n}}.$$
and also 
$$ \max_{v\in V, w_1, w_2\in V\setminus\{v\}}| H_{w_1v} -  H_{w_2v}| \lesssim 1.$$
In both results, the implicit constants depend on $p$ only.
\end{lemma}
\begin{proof} We know from \eqref{boundegree} that, with high likelihood, 
$$ \max_{w \in V} \left| \deg(w) - np \right| \lesssim \sqrt{n \log{n}}.$$
We first argue that $H_{wv} \leq C_p n$ for some constant $0 < C_p < \infty$. Let $X_0, X_1, \ldots, $ denote the simple random walk on the graph. We note that there are $|A| = np \pm \mathcal{O}(\sqrt{n \log{n}})$ vertices where the likelihood of hitting $v$ in the next step is at least 
$$ \mathbb{P}\left(X_{k+1} = v\big| X_k \in A\right) = \frac{1}{\deg(X_k)} = \frac{1}{np} \pm \mathcal{O}\left( \frac{\sqrt{\log{n}}}{n^{3/2}}\right).$$
If the random walk $X_k \in B$, then
$$ \mathbb{P}\left(X_{k+1} \in A\big| X_k \in B\right) \geq p + o(1).$$
We can now do a simple case distinction: either $X_k \in A$ in which case $X_{k+1} = v$ with likelihood at least $1/(np)$ or $x_k \in B$ in which case $X_{k+2} = v$ with likelihood at least $1/n$. Using a geometric probability distribution for stochastic domination, we see that  $H_{wv} \leq C_p n$. \\

We can now consider an arbitrary starting vertex $w \in V$ together with a random walk of 3 steps started in $w$. Let $T_{wv}$ be the time that it takes to hit $v$ starting from $w$, so that $\mathbb E[T_{wv}]=H_{wv}$. There are two cases: the random walk happens to hit $v$ within the first $3$ steps or it does not. Therefore, if $\mu_3$ denotes the distribution after $3$ steps,  
\begin{align*}
H_{wv} &= \mathcal O(1) + \sum_{u \in V} \mu_3(u) \cdot \mathbb E[T_{uv} \cdot \mathbb E[1_{T_{w v}>3}|X_3=u]].
\end{align*}
Note that $$\mathbb E[T_{uv} \cdot \mathbb E[1_{T_{w v}>3}|X_3=u]]\leq H_{uv}\leq C_p n.$$
The same bound holds if we replace $w$ with the stationary distribution. Using
\begin{align*}
\|\mu_3-\pi\|_1\lesssim \frac{\log{n}}{n},
\end{align*}
which follows from \eqref{specgap} with $k=3$ we deduce that with high probability and uniformly over all $v\in V$ and $w_1, w_2\in V\setminus\{v\}$
$$ |H_{w_1v} - H_{w_2v}| \lesssim \log{n}.$$
Note that, by convexity, the result remains true if $w_2$ is replaced by an arbitrary probability measure. If we choose $\nu$ from Lemma \ref{exact}, we deduce
\begin{align*}
\left|H_{wv}-\frac{2E}{\deg v}+1\right|\lesssim \log{n}, 
\end{align*}
and thus we deduce since, using \eqref{boundegree}, that
\begin{align*}
\left|\frac{2E}{\deg v}-n\right|=\mathcal O\left(\sqrt {n\log n}\right)
\end{align*}
with high probability.
\end{proof}

\subsection{Concentration of Hitting Times in $B$.} We can now use these results to establish concentration of the hitting times in $B$.  For this purpose, let us keep our original arbitrary vertex $w \in B$ considered above and let us additionally consider another vertex $w_2 \in B$.  The distribution of a random walk started in $w_2$ after 3 steps will be denoted by $\mu_2$. Using \eqref{eq:three} once in $w$ and once in $w_2$ and invoking the uniform asymptotics \eqref{eq:four} we arrive at
\begin{align*}
  H_{wv} -  H_{w_2v}    &=    \mathcal{O}\left(\frac{\sqrt{\log{n}}}{\sqrt{n}}\right) +  \sum_{c \in V}  (\mu(c) -  \mu_2(c)) \cdot H_{cv} 
\end{align*}
Lemma \ref{lem:1}, the fact that two probability distributions $\mu$ and $\mu_2$ have equal $\ell^1-$norm and the Cauchy-Schwarz inequality allow us to rewrite this as
\begin{align*}
\left|  H_{wv} -  H_{w_2v}  \right|  &=    \left| \mathcal{O}\left(\frac{\sqrt{\log{n}}}{\sqrt{n}}\right) +  \sum_{c \in V}  (\mu(c) -  \mu_2(c)) \cdot (n + (H_{cv} - n))  \right|\\
  &\leq \mathcal{O}\left(\frac{\sqrt{\log{n}}}{\sqrt{n}}\right) +  \sum_{c \in V}  \left| \mu(c) -  \mu_2(c) \right| \cdot  \left| H_{cv} - n\right| \\
  &\lesssim  \mathcal O\left(\frac{\sqrt{\log{n}}}{\sqrt{n}}\right) + \sqrt{n \log{n}}\cdot \| \mu - \mu_2 \|_{\ell^1}.
\end{align*}
Using a triangular inequality and \eqref{specgap} with $k=3$, we obtain
$$ \| \mu - \mu_2 \|_{\ell^1} \lesssim \frac{\log{n}}{n},$$
leading to 
\begin{align}\label{concentrationforb}
\left|  H_{wv} -  H_{w_2v}  \right|  &\lesssim  \frac{(\log{n})^{3/2}}{\sqrt{n}}.
\end{align}

\subsection{Concentration for vertices in $A$.} We will now adapt the argument for vertices in $A$. Let $w \in A$ and consider the usual random walk in $A$ after 3 steps. We would like to argue analogously as before to obtain a combinatorial formula for the hitting time. There are now two types of mistakes that can happen: those that arise if the random walk already ends up in $v$ after 1 step or after 2 steps.
\begin{enumerate}
\item \textit{Scenario 1.} The random walk moves from $w$ to $v$ in the first step. In that case the true hitting time is 1 where as the naive formula \eqref{eq:two} produces an error given by the average hitting time in a point after 2 random steps started in $v$ (as before, this quantity is $\mathcal O(n)$ and independent of $w$).
\item  \textit{Scenario 2.} The random walk arrives in $v$ at the second step (but not the first). In that case, we deduce that the first random step has to lead from $w$ to another vertex in $A$ and then from that vertex to $v$. The error made in the naive formula is that it produces the average hitting time in $A$ as opposed to the true hitting time, which equals to $2$. 
\end{enumerate}
We deduce that there are two numbers $O_1, O_2=\mathcal O(n)$, independent of $w\in A$, with
\begin{align} 
 H_{wv} &= 3 + \sum_{a \in V} \mu(a) \cdot H_{av} \\
 &+  \mathbb{P}\left( \mbox{Scenario 1}\right) \cdot O_1 +  \mathbb{P}\left( \mbox{Scenario 2}\right) \cdot O_2. \nonumber
\end{align}
Note that $O_1$ is simply the expected hitting time weighted by the outcome of a random walk of length 2 started in $v$ and $O_2$ is the expected hitting time weighted by the distribution of a random walk of length 1 started in $v$ (which, as above, is merely the average hitting time in $A$).
It remains to show that, just as above, the likelihoods of the two scenarios depend on $w$ in a weak sense. We start with Scenario 1.  We have, using \eqref{boundegree},
$$  \mathbb{P}\left( \mbox{Scenario 1}\right) = \frac{1}{\deg(w)} = \frac{1}{np \pm \mathcal{O}(\sqrt{n \log{n}})} = \frac{1}{np} + \mathcal{O}\left( \frac{\sqrt{\log{n}}}{n^{3/2}}\right).$$
The likelihood of Scenario 2 was already implicitly computed above: there we computed the likelihood of having a random walk from a vertex in $B$ travel to a vertex in $A$ and then to $v$. However, that likelihood is independent of whether one starts in $A$ or in $B$ and the very same argument as above (see \eqref{eq:four}) gives
\begin{align} 
  \mathbb{P}\left( \mbox{Scenario 2}\right) = \frac{|A|}{n^2 p}  +  \mathcal{O}\left(\frac{\sqrt{\log{n}}}{ n^{3/2}}\right)  \nonumber
\end{align}
The remainder of the argument is exactly the same as above and we deduce that
\begin{align}\label{concentrationfora}
\max_{w_1, w_2 \in A} |H_{w_1v} -H_{w_2v} | \lesssim \frac{(\log{n})^{3/2}}{\sqrt{n}}.
\end{align}

\subsection{Proof of the Theorem}
At this point, we know that the expected hitting time in both $A$ and $B$ is essentially constant, owing to \eqref{concentrationfora} and \eqref{concentrationforb}, up to an error of size $\lesssim (\log n)^{3/2}/\sqrt{n}$. We define the average hitting times in $A$ and in $B$
\begin{align*}
H_A=\frac{1}{|A|}\sum_{w\in A}H_{wv}, \quad H_B=\frac{1}{|B|}\sum_{w\in B}H_{wv}.
\end{align*}
From Lemma \eqref{exact}, we know
\begin{equation}\label{baba}
H_A=\frac{2E}{\deg v}-1,  
\end{equation}
and, appealing to \eqref{concentrationfora}, we obtain that for $w\in A$
\begin{align*}
H_{wv}=\frac{2E}{\deg v}-1+\mathcal O\left(\frac{(\log n)^{3/2}}{\sqrt n}\right).
\end{align*}
It remains to show that $$H_{B} - H_{A} = \frac{1}{p} +   \mathcal{O}\left( \frac{(\log{n})^{3/2}}{\sqrt{n}} \right),$$ 
for the result will follow from \eqref{concentrationforb} and \eqref{baba}. Let now $w \in B$ be an arbitrary vertex. Performing one step of a random walk in $w$, we deduce that
\begin{align*}
H_B &= H_{wv} + \mathcal{O}\left( \frac{(\log{n})^{3/2}}{\sqrt{n}} \right) \\
&= \mathcal{O}\left( \frac{(\log{n})^{3/2}}{\sqrt{n}} \right) + 1 +  \frac{1}{\deg(w)} \sum_{z \in N_A(w)} H_{zv} +  \frac{1}{\deg(w)} \sum_{z \in N_B(w)} H_{zv}.
 \end{align*}
At this point we can invoke the strong concentration of hitting times in $A$ and $B$ once more to deduce that

\begin{align*}
H_B &= \mathcal{O}\left( \frac{(\log{n})^{3/2}}{\sqrt{n}} \right) + 1 +  \frac{1}{\deg(w)} \sum_{z \in N_A(w)} H_A + \frac{1}{\deg(w)} \sum_{z \in N_B(w)} H_B \\
&=  \mathcal{O}\left( \frac{(\log{n})^{3/2}}{\sqrt{n}} \right) + 1 +  \frac{\# N_A(w)}{\deg(w)} H_A + \frac{\# N_B(w)}{\deg(w)} H_B.
 \end{align*}

Vertices in $B$ only have neighbors in $A$ and in $B$ and therefore, for $w \in B$,
$$ 1 = \frac{\#N_A(w) + \# N_B(w)}{\deg(w)}$$
from which we deduce
$$ \frac{\# N_A(w)}{\deg(w)} H_B =   \mathcal{O}\left( \frac{(\log{n})^{3/2}}{\sqrt{n}} \right) + 1 +  \frac{\# N_A(w)}{\deg(w)} H_A$$
and thus
$$ H_B =   \mathcal{O}\left( \frac{(\log{n})^{3/2}}{\sqrt{n}} \right) + \frac{\deg(w)}{ \# N_A(w)} + H_A.$$
We also note that, from \eqref{boundegree},
\begin{align*}
\frac{\deg(w)}{\# N_A(w)} &= \frac{np + \mathcal{O}(\sqrt{n \log{n}})}{ p |A| +  \mathcal{O}(\sqrt{n \log{n}})} \\
&=  \mathcal{O}\left( \frac{\sqrt{\log{n}}}{\sqrt{n}} \right)  + \frac{n}{|A|} = \frac{1}{p} + \mathcal{O}\left( \frac{(\sqrt{\log{n}}}{\sqrt{n}} \right).
\end{align*}
Thus
\begin{align*} \label{eq:diff}
H_B - H_A = \frac{1}{p} +   \mathcal{O}\left( \frac{ (\log{n})^{3/2}}{\sqrt{n}} \right).
\end{align*}

\subsection{A Central Limit Theorem} We conclude by showing how the Theorem immediately implies a central limit theorem for the hitting times starting from any point. It can be easily generalized to an arbitrary initial configuration that does not include the target point. This follows from the Theorem and the fact that the main fluctuation comes from the fluctuation of the degree of the target vertex while being virtually independent (up to fluctuations of size $\mathcal{O}(1)$) of everything else.

\begin{corollary}
Fix two vertices $1, 2$. Then, the distribution of $H_{12}$ over all Erd\H{o}s-R\'enyi random graphs satisfies, as $n \rightarrow \infty$,
\begin{align*}
\sqrt{\frac{p}{n(1-p)}}\left(H_{12}-n\right)\Rightarrow \mathcal N(0,1)
\end{align*}
\end{corollary}
\begin{proof}
The central limit theorem for a binomial random variable says that 
\begin{align*}
\deg(2)=np\left(1-\sqrt{\frac{1-p}{pn}}Z_n\right), 
\end{align*}
where $Z_n$ converges in distribution to a normal random variable. A standard concentration bound for the number of edges gives
\begin{align*}
2 |E|=n^2p\left(1+\mathcal O\left(\frac{\sqrt{\log n}}{n}\right)\right)
\end{align*}
with high probability. Combining with our main theorem, we obtain
$$
H_{12}=n\frac{1+\mathcal O\left(\frac{\sqrt{\log n}}{n}\right)}{1-\sqrt{\frac{1-p}{np}}Z_n}+\mathcal O\left(1\right),
$$
where the error terms are uniform with respect to the measures $\mu_n$. 
Therefore
\begin{align*}
\sqrt{\frac{p}{n(1-p)}}\left(H_{12}-n\right)&=\sqrt{\frac{np}{1-p}}\left[\frac{1+\mathcal O\left(\frac{\sqrt{\log n}}{n}\right)}{1-\sqrt{\frac{1-p}{np}}Z_n}-1+\mathcal O\left(\frac{1}{n}\right)\right]\\&=\frac{Z_n+\mathcal O\left(\frac{\sqrt{\log n}}{\sqrt n}\right)}{1+\mathcal O\left(\frac{1}{\sqrt n}\right)}+\mathcal O\left(\frac{1}{\sqrt n}\right).
\end{align*}
The result then follows from Slutsky's theorem. 
\end{proof}

\subsection{Proof of the Proposition}
The Proposition will be proven by induction. The statement is very easy to prove when $k=1$. For the induction step $k \rightarrow k+1$, we make use of a technical Lemma.

\begin{lemma} \label{lem:tech}
Let $v \in \mathbb{R}^n$ be a row vector whose entries sum to 0. Then
$$ \| v D^{-1}A \|_{\ell^2} \leq c_{p} \frac{\sqrt{\log{n}}}{\sqrt{n}} \|v\|_{\ell^2}.$$
\end{lemma}
\begin{proof} We use two different ingredients. The first is to write the diagonal matrix $D^{-1}$ as the sum of two diagonal matrices $D^{-1} = D_1 + D_2$ using
$$ \frac{1}{\deg(v)} = \frac{1}{np} + \left( \frac{1}{\deg(v)} - \frac{1}{np} \right).$$
 $D_1$ is a matrix with entries $1/(np)$ on the diagonal and $D_2 = D^{-1} - D_1$. Note that $D_1$ is a multiple of the identity and $\|D_1\| = 1/(np)$.
Since
$$ \deg(v) = np \pm \mathcal{O}(\sqrt{n \log{n}}),$$
we have
$$ \left|  \frac{1}{\deg(v)} - \frac{1}{np}  \right| \lesssim \frac{\sqrt{\log{n}}}{n^{3/2}} \qquad \mbox{and thus} \qquad \|D_2\| \lesssim  \frac{\sqrt{\log{n}}}{n^{3/2}}.$$
The second ingredient are the spectral properties of $A$. $A$ has one eigenvalue at scale $\lambda_1 \sim np$ while the second largest eigenvalue of $A$ satisfies $|\lambda_2(A)| \lesssim_p \sqrt{n}$ with probability tending to 1 (see  F\"uredi-Komlos \cite{fur}). This allows us to write, using that $D_1$ is a multiple of the identity matrix and thus commutes with all other matrices,
\begin{align*}
 \| vD^{-1}A \|&= \| v(D_1 + D_2)A v\| \\
 &\leq \| vD_1 A\| + \| v D_2 A\| \leq \| vD_1 A \| + \| A\| \| D_2\| \|v\|  \\
 &\lesssim \frac{\|vA \|}{n} +  \frac{\sqrt{\log{n}}}{n^{3/2}} \|A\| \|v\| \leq \frac{\|vA \|}{n} +  \frac{\sqrt{\log{n}}}{\sqrt{n}}  \|v\|.
 \end{align*}
 It remains to analyze the first term. Since $A$ is symmetric, the spectral theorem applies. Using $\phi$ to denote the $\ell^2-$normalized eigenvector of $A$ associated to the eigenvalue $\lambda_1$ (and $\lambda_2$ the second largest eigenvalue in absolute value), we have
 $$ \|vA \|^2 \leq \lambda_1(A)^2 \left\langle \phi, v \right\rangle^2 + \lambda_2(A)^2 \|v\|^2.$$
 Now we use a result of Mitra \cite{mitra} telling us that the eigenvector associated to the adjacency matrix of an Erd\H{o}s-Renyi random graph $G(n,p)$ and $0 < p < 1$ fixed is nearly constant and
 $$ \max_{1 \leq i \leq n} \left| \phi_i - \frac{1}{\sqrt{n}} \right| \leq c_p \frac{\sqrt{\log{n}}}{n}.$$
 This allows us to write
 \begin{align*}
  \left\langle \phi, v \right\rangle = \sum_{i=1}^{n} \phi_i v_i =  \sum_{i=1}^{n} \frac{1}{\sqrt{n}} v_i +   \sum_{i=1}^{n}  \left(\phi_i - \frac{1}{\sqrt{n}}\right) v_i.
 \end{align*}
 The first sum vanishes because $v$, by assumption, has mean value 0. Using the Cauchy-Schwarz inequality 
  $$  \left\langle \phi, v \right\rangle^2 = \left(\sum_{i=1}^{n}  \left(\phi_i - \frac{1}{\sqrt{n}}\right) v_i\right)^2 \lesssim_p \frac{\log{n}}{n} \cdot \|v\|_{\ell^2}^2.$$
 Therefore
 \begin{align*}
 \frac{\|vA \|_{\ell^2}}{n} &\lesssim \frac{1}{n} \sqrt{ \lambda_1(A)^2 \left\langle \phi, v \right\rangle^2 + \lambda_2(A)^2 \|v\|^2}\\
&\lesssim \frac{1}{n} \sqrt{ n^2 \left\langle \phi, v \right\rangle^2 + n\|v\|^2} \\
 &\lesssim \frac{1}{n}\sqrt{n^2 \frac{\log{n}}{n}\|v\|^2 + n \|v\|^2} \lesssim \frac{\sqrt{\log{n}}}{\sqrt{n}} \|v\|_{\ell^2}.
 \end{align*}
 This establishes the desired result.
 \end{proof}

\begin{proof}[Proof of the Proposition]
$\mu_1$ is easy to describe, it assumes values $1/\deg(v)$ in the neighbors of $v$ and value 0 everywhere else. Thus
$$\| \mu_1 - \pi\|_{\ell^2} \leq \| \mu_1\|_{\ell^2} + \| \pi \|_{\ell^2} \lesssim \frac{1}{\sqrt{n}} + \| \pi \|_{\ell^2} \lesssim \frac{1}{\sqrt{n}}.$$
We can now argue via induction. Suppose the desired statement holds for $\mu_k$. Then, using the fact that $\pi D^{-1} A = \pi$, we have
$$\mu_{k+1} - \pi = \mu_{k}D^{-1}A - \pi = \left( \mu_{k} - \pi \right)D^{-1}A.$$
We observe that $\mu_k$ and $\pi$ are probability measures, all the arising vectors always have mean value 0. Then
$$ \| \mu_{k+1} - \pi \|_{\ell^2} =   \| \left( \mu_{k}D^{-1}A - \pi \right) \|_{\ell^2}.$$
Since both $\mu_k$ and $\pi$ are probability distributions, we have that $\mu_k - \pi$ has mean value 0 and Lemma \ref{lem:tech} applies to give
$$   \| \left( \mu_{k} - \pi \right)D^{-1}A\|_{\ell^2} \lesssim c_p \frac{\sqrt{\log{n}}}{\sqrt{n}} \|\mu_{k} - \pi \|_{\ell^2}$$
from which the result follows.
\end{proof}

\textbf{Acknowledgment.} The authors are grateful for discussions with Karel Devriendt and grateful to Felix Joos, Jonathan Schrodt and Tom Stalljohann for spotting a gap in a previous version of the manuscript.


\begin{thebibliography}{10}


\bibitem{chu} F. Chung, L. Linyuan, The diameter of sparse random graphs. Advances in Applied Mathematics 26.4 (2001): 257-279.

\bibitem{gestalt} A. Desolneux, L. Moisan, and J. Morel. From gestalt theory to image analysis: a probabilistic approach. Vol. 34. Springer Science \& Business Media, 2007.

\bibitem{persi} P. Diaconis, L. Miclo. On quantitative convergence to quasi-stationarity. Annales de la Faculte des sciences de Toulouse: Mathematiques. Vol. 24. No. 4. 2015.

\bibitem{fur} Z. F\"uredi and J. Komlos, The eigenvalues of random symmetric matrices, Combinatorica 1 (1981), p. 233--241.

\bibitem{hel} A. Helali and M. L\"owe, 
Hitting times, commute times, and cover times for random walks on random hypergraphs.
Statist. Probab. Lett. 154 (2019), 108535, 6 pp.





\bibitem{kir} G. Kirchhoff,
\"Uber die Aufl\"osung der Gleichungen, auf welche man bei der Untersuchung der linearen Vertheilung
galvanischer Str\"ome gef\"uhrt wird. Ann. Phys. Chem. 72 (1847), 497–508

\bibitem{klee} V. Klee and D. Larman, Diameters of random graphs, Canadian Journal of Mathematics 33.3 (1981): 618-640.

\bibitem{lov} L. Lov\'asz,
Random walks on graphs: a survey. (English summary) Combinatorics, Paul Erdős is eighty, Vol. 2 (Keszthely, 1993), 353--397,
Bolyai Soc. Math. Stud., 2, János Bolyai Math. Soc., Budapest, 1996.

\bibitem{lowe} M. L\"owe and F. Torres, On hitting times for a simple random walk on dense Erd\H{o}s-R\'enyi random graphs. Statistics \& Probability Letters 89 (2014), p. 81--88.

\bibitem{lowe2} M. L\"owe and S. Terveer, A Central Limit Theorem for the average target hitting time for a random walk on a random graph. arXiv preprint arXiv:2104.01053.


\bibitem{von} U. von Luxburg, A. Radl and M. Hein, Hitting and commute times in large random neighborhood graphs. The Journal of Machine Learning Research, 15 (2014), 1751--1798.

\bibitem{lyo} R. Lyons, Y. Peres, Probability on trees and networks.  Cambridge University Press, 2017.

\bibitem{mitra} P. Mitra, Entrywise bounds for eigenvectors of random graphs. the electronic journal of combinatorics, R131 (2009).


\bibitem{sood} V. Sood, S. Redner and D. ben-Avraham, 
First-passage properties of the Erdős-Renyi random graph. 
J. Phys. A 38 (2005), no. 1, 109--123.

\bibitem{sylvester} J. Sylvester, John
Random walk hitting times and effective resistance in sparsely connected Erd\H{o}s-R\'enyi random graphs. 
J. Graph Theory 96 (2021), no. 1, 44--84.

\end{thebibliography}
\end{document}